\documentclass[final,11pt,twocolumn]{elsarticle}

\usepackage{verbatim,a4wide}

\usepackage{amsmath}
\usepackage{amssymb}
\usepackage{amsthm}

\usepackage{bm}

\usepackage{algorithm}
\usepackage{algpseudocode}

\usepackage[cp1250]{inputenc}
\usepackage[OT4]{fontenc}
\usepackage[english]{babel}

\numberwithin{equation}{section}
\numberwithin{algorithm}{section}
\numberwithin{figure}{section}
\numberwithin{table}{section}

\newtheorem{theorem}{Theorem}[section]
\newdefinition{remark}[theorem]{Remark}
\newdefinition{example}[theorem]{Example}

\newcommand{\p}[1]{\mbox{\textsf{#1}}}

\journal{Elsevier} 

\begin{document}

\begin{frontmatter}

\title{Linear-time geometric algorithm for evaluating B\'{e}zier curves}

\author{Pawe{\l} Wo\'{z}ny\corref{cor}}
\ead{Pawel.Wozny@cs.uni.wroc.pl}
\cortext[cor]{Corresponding author. Fax {+}48 71 3757801}

\author{Filip Chudy}
\ead{Filip.Chudy@cs.uni.wroc.pl}

\address{Institute of Computer Science, University of Wroc{\l}aw,
         ul.~Joliot-Curie 15, 50-383 Wroc{\l}aw, Poland}

\begin{abstract}
A new algorithm for computing a point on a polynomial or rational curve in
B\'{e}zier form is proposed. The method has a~geometric interpretation and uses 
only convex combinations of control points. The new algorithm's computational
complexity is linear with respect to the number of control points and its
memory complexity is $O(1)$. Some remarks on similar methods for surfaces in
rectangular and triangular B\'{e}zier form are also given.
\end{abstract}


\end{frontmatter}

\section{Introduction}                                  \label{S:Introduction}

Let $b_k:D\rightarrow\mathbb R$ $(k=0,1,\ldots,N;\, N\in\mathbb N)$ be
real-valued multivariable functions such that
\begin{equation}\label{E:BasisF}
b_k(\bm{t})\geq0,\quad \sum_{k=0}^{N}b_k(\bm{t})\equiv 1
\end{equation} 
for $\bm{t}\in C\subseteq D$.

Let us define the \textit{rational parametric object}
$\p{S}_N:C\rightarrow \mathbb E^d$ $(d\in\mathbb N)$ by
\begin{equation}\label{E:ParObject}
\p{S}_N(\bm{t}):=\frac{\displaystyle \sum_{k=0}^{N}\omega_k\p{W}_kb_k(\bm{t})}
                    {\displaystyle \sum_{k=0}^{N}\omega_k b_k(\bm{t})}
\end{equation}
with the \textit{weights} $\omega_k>0$, and \textit{control points}
$\p{W}_k\in\mathbb E^d$ $(0\leq k\leq N)$. If
$\omega_0=\omega_1=\ldots=\omega_N$, then 
$$
\p{S}_N(\bm{t})=\sum_{k=0}^{N}\p{W}_kb_k(\bm{t}).
$$

In the sequel, we prove that for a given $\bm{t}\in C$, the point
$\p{S}_N(\bm{t})\in\mathbb E^d$ can be computed by Algorithm~\ref{A:GenAlg}. 

\begin{algorithm}[ht!]
\caption{Computation of $\p{S}_N(\bm{t})$}\label{A:GenAlg}
\begin{algorithmic}[1]
\Procedure {GenAlg}{$N, \bm{t}, \omega, \p{W}$}


\State $h_0 \gets 1$
\State $\p{Q}_0 \gets \p{W}_0$


\For{$k \gets 1,N$}
       \State $h_k \gets 
           \left(1+\frac{\displaystyle \omega_{k-1}b_{k-1}(\bm{t})}
                        {\displaystyle h_{k-1}\omega_kb_k(\bm{t})}\right)^{-1}$
       \Statex                                                      
       \State $\p{Q}_k \gets (1-h_k)\p{Q}_{k-1}+h_k\p{W}_k$
\EndFor


\State \Return $\p{Q}_N$


\EndProcedure
\end{algorithmic}
\end{algorithm}

\begin{remark}\label{R:DivZero}
Let us fix $\bm{t}\in C$. Suppose that there exists $1\leq k\leq N$ such that
$b_k(\bm{t})=0$. Then one has the division by $0$ in the line 5 of 
Algorithm~\ref{A:GenAlg}. Such \textit{special cases} should be considered
separately. Observe that it is always possible because at least for one 
$0\leq j\leq N$ we have $b_j(\bm{t})>0$ (cf.~\eqref{E:BasisF}).
\end{remark}

\begin{theorem}\label{T:GenThm}
The quantities $h_k$ and $\p{Q}_k$ $(0\leq k\leq N)$ computed by 
Algorithm~\ref{A:GenAlg} have the following properties: 

\begin{enumerate}
\itemsep1ex

\item $h_k\in[0,1]$,

\item $\p{Q}_k\in\mathbb E^d$, 

\item 
      $\p{Q}_k\in
      C_k\equiv\mbox{conv}\{\p{W}_0,\p{W}_1,\ldots,\p{W}_k\}$
      (i.e., $\mbox{conv}\{\p{Q}_0,\p{Q}_1,\ldots,\p{Q}_k\}\subseteq C_k$).

\end{enumerate}
Moreover, $\p{S}_N(\bm{t})=\p{Q}_N$.
\end{theorem}
\begin{proof}
Let us define
$$
h_k:=\frac{\displaystyle \omega_kb_k(\bm{t})}
          {\displaystyle \sum_{j=0}^{k}\omega_j b_j(\bm{t})},\quad                                                       
\p{Q}_k:=\frac{\displaystyle \sum_{j=0}^{k}\omega_j\p{W}_jb_j(\bm{t})}
              {\displaystyle \sum_{j=0}^{k}\omega_jb_j(\bm{t})}
$$
$(k=0,1,\ldots,N)$. 

It is clear that $h_k\in[0,1]$, $\p{Q}_k\in{\mathbb E^d}$ for $0\leq k\leq N$,
$h_0=1$, $\p{W}_0=\p{Q}_0$, and $\p{S}_N(\bm{t})=\p{Q}_N$. Certainly,
$$
\p{Q}_k\in\mbox{conv}\{\p{W}_0,\p{W}_1,\ldots,\p{W}_k\}\quad (0\leq k\leq N).
$$
To end the proof, it is enough to check that:
$$
\left\{
\begin{array}{l}
(1-h_k)\p{Q}_{k-1}+h_k\p{W}_k=\p{Q}_k,\\[1ex]
\omega_kb_k(\bm{t})h_k^{-1}=
         \omega_{k-1}b_{k-1}(\bm{t})h_{k-1}^{-1}+\omega_kb_k(\bm{t})
\end{array}
\right.
$$         
for $1\leq k\leq N$ (cf.~lines 5, 6 in Algorithm~\ref{A:GenAlg}).
\end{proof}

Let us notice that Algorithm~\ref{A:GenAlg} has a geometric interpretation, 
uses only convex combinations of control points of $\p{S}_N$ and has linear
complexity with respect to $N$ --- under the assumption that all quotients of
two consecutive basis functions can be computed in the total time~$O(N)$. 

\begin{remark}
It may be worth mentioning that
\begin{equation}\label{E:Rel_1-h_k}
1-h_k=\frac{h_k}{h_{k-1}}\frac{\omega_{k-1}b_{k-1}(\bm{t})}
                              {\omega_kb_k(\bm{t})}                              
\end{equation}
for $1\leq k\leq N$. Using this simple relation, one can propose 
a subtraction-free version of Algorithm~\ref{A:GenAlg}. Such formulation can be
important for numerical reasons (cf.~the problem of \textit{cancellation of
digits}; see, e.g., \cite[\S2.3.4]{DB2008}). 
\end{remark}

We use relation~\eqref{E:Rel_1-h_k} in the proof of the following theorem
which shows an important property of Algorithm~\ref{A:GenAlg}.  

\begin{theorem}\label{T:GenConvexHullThm}
Let us fix $\bm{t}, \bm{u}\in C$. Assume that the numbers $h_k$ 
$(1\leq k\leq N)$ computed by Algorithm~\ref{A:GenAlg} are non-zero.
Suppose that
$$
\frac{b_k(\bm{t})}{b_{k+1}(\bm{t})}\leq \frac{b_k(\bm{u})}{b_{k+1}(\bm{u})}
                                                      \quad (0\leq k\leq N-1).
$$
Then the point $\p{S}_N(\bm{u})\in{\mathbb E^d}$ is
in the convex hull of the points $\p{Q}_0,\p{Q}_1,\ldots,\p{Q}_{N}$
computed by Algorithm~\ref{A:GenAlg}.
\end{theorem}
\begin{proof}
Let the numbers $h_k$ and the points $\p{Q}_k$ $(0\leq k\leq N)$ be computed by
Algorithm~~\ref{A:GenAlg} for a fixed $\bm{t}\in C$. 

Using relation~\eqref{E:Rel_1-h_k} and the assumption that $h_k\neq 0$ 
$(1\leq k\leq N)$, observe that
$$
\p{W}_k=h_k^{-1}\p{Q}_k-h_{k-1}^{-1}\frac{\omega_{k-1}b_{k-1}(\bm{t})}
                                       {\omega_{k}b_{k}(\bm{t})}\p{Q}_{k-1}
$$
for $1\leq k\leq N$. Thus, after simple algebra, we obtain
\begin{eqnarray*}
\lefteqn{\p{S}_N(\bm{u})=D_N(\bm{u})^{-1}
     \Bigg(\frac{\omega_N}{h_N}b_N(\bm{u})\cdot\p{Q}_{N}}\\
&&\hspace{-1.5ex}+     
     \sum_{k=0}^{N-1}\frac{\omega_k}{h_k}b_k(\bm{u})
              \left(1-\frac{b_k(\bm{t})b_{k+1}(\bm{u})}
                           {b_{k+1}(\bm{t})b_{k}(\bm{u})}\right)\cdot\p{Q}_k                           
     \Bigg),
\end{eqnarray*}
where $D_N(\bm{u}):=\sum_{k=0}^{N}\omega_kb_k(\bm{u})>0$. 

Now, from our assumptions, it easily follows that the point $\p{S}_N(\bm{u})$
belongs to the set $\mbox{conv}\{\p{Q}_0,\p{Q}_1,\ldots,\p{Q}_N\}$, because
the $h_k$ $(0\leq k\leq N)$ are positive (cf.~Theorem~\ref{T:GenThm}).
\end{proof}

The main aim of this article is to use the presented results to propose a new
method for evaluating a polynomial or rational B\'{e}zier curve, which has
a geometric interpretation, linear complexity with respect to the number of 
control points, good numerical properties and computes only convex combinations
of points from ${\mathbb E^d}$. See Section~\ref{S:NewEvalBC}.

A similar approach can also be used for the evaluation of polynomial and 
rational tensorproduct, as well as triangular, B\'{e}zier surfaces. Some remarks 
on this issue are given, without \textit{technical} details and 
\textit{rigorous} algorithms, in Section~\ref{S:NewEvalBS}.

\section{New algorithm for evaluating B\'{e}zier curves}   \label{S:NewEvalBC}

Let there be given points $\p{W}_0, \p{W}_1, \ldots,\p{W}_n\in\mathbb{E}^d$
$(n,d\in\mathbb N)$. Let us consider the (polynomial) \textit{B\'{e}zier curve}
of the form
\begin{equation}\label{E:BezierCurve}
\p{P}_n(t):=\sum_{k=0}^{n}\p{W}_kB^n_k(t)\quad (t\in[0,1]),
\end{equation}
where $B^n_k$ is the $k$th \textit{Bernstein polynomial} of degree $n$,
\begin{equation}\label{E:BerPoly}
B^n_k(t):=\binom{n}{k}t^k(1-t)^{n-k}\quad (0\leq k\leq n).
\end{equation}

For a given $t\in[0,1]$, the point $\p{P}_n(t)\in\mathbb E^d$ can be computed
by famous the de Casteljau algorithm (see, e.g., \cite[\S4.2]{Farin2002} and
Appendix), which has good numerical properties, a simple geometric
interpretation and computes only convex combinations of control points
$\p{W}_k$ $(0\leq k\leq n)$. However, the computational complexity of this 
method is $O(dn^2)$, which makes it quite expensive.

Probably, the fastest way to compute the coordinates of the point 
$\p{P}_n(t)\in\mathbb E^d$ is to use the algorithm proposed in~\cite{SV1986} 
for evaluating a polynomial $p$ given in the form
$$
p(t):=\sum_{k=0}^{n}p_kt^k(1-t)^{n-k}\quad (p_k\in\mathbb R)
$$
$d$ times (once for each dimension). This method has $O(dn)$ computational
complexity and $O(1)$ memory complexity. It uses the concept of Horner's rule 
(see, e.g., \cite[Eq.~(1.2.2)]{DB2008}). 

Note that some other methods for evaluating a~B\'{e}zier curve are
also known. See, e.g., \cite{Bezerra2013} or \cite{Peters1994}, where the case 
of B\'{e}zier surfaces was also studied (cf.~Section~\ref{S:NewEvalBS}), and
papers cited therein. 

Let $\p{R}_n$ be a \textit{rational B\'{e}zier curve} in $\mathbb E^d$,
\begin{equation}\label{E:RatBezierCurve}
\p{R}_n(t):=\frac{\displaystyle \sum_{k=0}^{n}\omega_k\p{W}_kB^n_k(t)}
                 {\displaystyle \sum_{k=0}^{n}\omega_kB^n_k(t)}     
                                                      \quad (t\in[0,1])
\end{equation}
with the weights $\omega_0, \omega_1,\ldots, \omega_n\in \mathbb R_{+}$. 
To compute the point $\p{R}_n(t)\in\mathbb E^d$ for a given $t\in[0,1]$, one 
can use the rational de Casteljau algorithm (see, e.g., 
\cite[\S13.2]{Farin2002} and Appendix), which also has $O(dn^2)$ computational
complexity, good numerical properties, a geometric interpretation and computes
only convex combinations of the control points $\p{W}_k$ $(0\leq k\leq n)$, or
use the idea from~\cite{SV1986}, which leads to linear-time method at the cost 
of losing some geometric properties.

The main purpose of this section is to propose a new efficient method
for computing a point on a B\'{e}zier curve and on a rational B\'{e}zier 
curve. The given algorithm has:
\begin{enumerate}
\itemsep0.15ex

\item a geometric interpretation,

\item quite good numerical properties, i.e., they are safe for floating-point
      computations,

\item linear computational complexity, i.e., $O(dn)$, and $O(1)$ memory 
      complexity,

\end{enumerate}
and computes only
\begin{enumerate}
\itemsep0.15ex
\setcounter{enumi}{3}

\item convex combinations of control points.

\end{enumerate}
As we show later, the new method combines the advantages of de Casteljau
algorithms and the low complexity of methods based on~\cite{SV1986}.

\subsection{New method}                                 \label{SS:NewMethodBC}

Let $\p{R}_n$ be the rational B\'{e}zier curve~\eqref{E:RatBezierCurve}.
Let us fix: a parameter $t\in[0,1]$, a natural number $n$, weights 
$\omega_0,\omega_1,\ldots,\omega_n>0$ and control points 
$\p{W}_0, \p{W}_1, \ldots,\p{W}_n\in\mathbb{E}^d$ $(d\in\mathbb N)$. 

Let the quantities $h_k$ and $\p{Q}_k$ $(0\leq k\leq n)$ be computed
recursively by formulas
\begin{equation}\label{E:Def_h_k_Q_k}
\left\{\hspace{-0.81ex}
\begin{array}{l}
h_0:=1,\quad \p{Q}_0:=\p{W}_0,\\[1ex]
h_k:=\frac{\displaystyle \omega_kh_{k-1}t(n-k+1)}
          {\displaystyle \omega_{k-1}k(1-t)+\omega_kh_{k-1}t(n-k+1)},\\[2.5ex]
\p{Q}_k:=(1-h_k)\p{Q}_{k-1}+h_k\p{W}_k  
\end{array}
\right.
\end{equation}
for $k=1,2,\ldots,n$.
 

\begin{theorem}\label{T:Thm1}
For all $k=0,1,\ldots,n$, the quantities $h_k$ and $\p{Q}_k$ satisfy:
\begin{enumerate}
\itemsep1ex

\item $h_k\in[0,1]$,

\item $\p{Q}_k\in\mathbb E^d$, 

\item 
       $\p{Q}_k\in
       C_k\equiv\mbox{conv}\{\p{W}_0,\p{W}_1,\ldots,\p{W}_k\}$
       (i.e., $\mbox{conv}\{\p{Q}_0,\p{Q}_1,\ldots,\p{Q}_k\}\subseteq C_k$).

\end{enumerate}
Moreover, we have $\p{R}_n(t)=\p{Q}_n$. 
\end{theorem}
\begin{proof}
The proof goes in a similar way to that of 
Theorem~\ref{T:GenThm}, where $N:=n$, $b_k(\bm{t}):=B^n_k(t)$.

Note that this method is robust --- \textit{special cases} $t=0$ and $t=1$ 
do not cause division by zero (cf.~Remark~\ref{R:DivZero}) and yield $\p{W}_0$
and $\p{W}_n$, respectively.
\end{proof}

In each step of the new method, the point $\p{Q}_k$, which is a convex
combination of points $\p{Q}_{k-1}$ and $\p{W}_k$, is computed. The last point
$\p{Q}_n$ is equal to the point $\p{R}_n(t)$. Thus, we obtain the new
linear-time geometric algorithm for computing a point on a rational B\'{e}zier
curve which computes only convex combinations of control points. For efficient
implementations, see Section~\ref{SS:ImplementationCostBC}.	

Note that if all weights $\omega_k$ are equal then $\p{Q}_n=\p{P}_n(t)$
(cf.~\eqref{E:BezierCurve}) --- the new method can also be used to evaluate 
a polynomial B\'{e}zier curve.

Figure~\ref{F:Figure1} illustrates the new method in case of a~planar 
polynomial B\'{e}zier curve of degree $n=5$.

\begin{figure*}[ht!]
\centering
\vspace*{-8.2ex}%
\includegraphics[angle=0,scale=0.5]{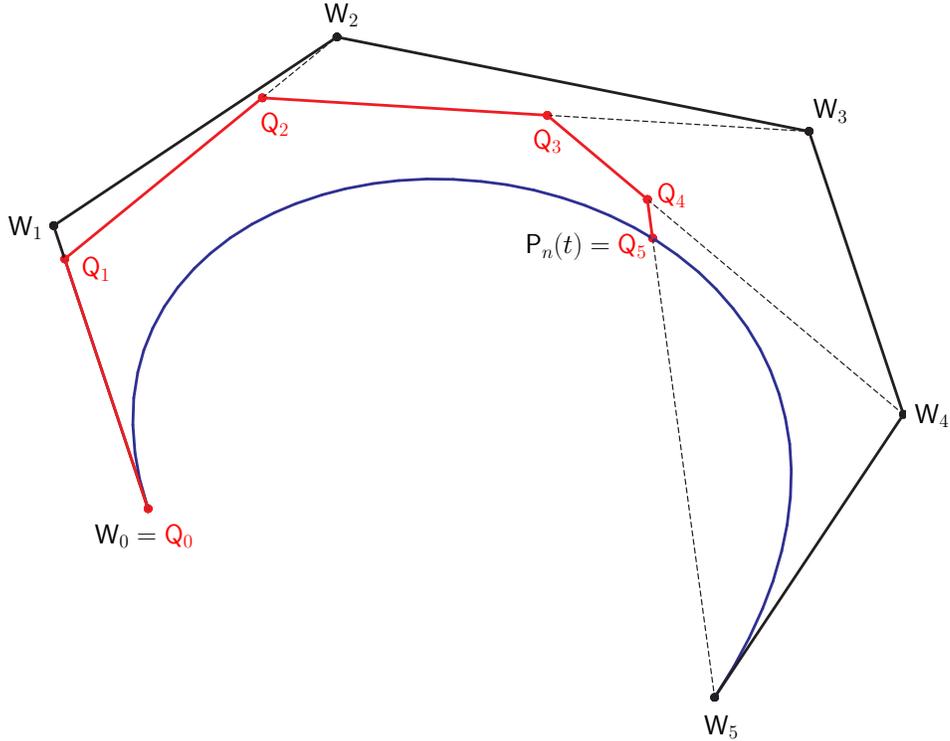}
\caption{Computation of a point on a planar polynomial B\'{e}zier curve of 
degree $n=5$ using the new method.}\label{F:Figure1}
\end{figure*}	

Using Theorem~\ref{T:GenConvexHullThm}, one can prove the following result
which tells even more about geometric properties of the new method. 

\begin{theorem}
Let the numbers $h_k$ and the points $\p{Q}_k$ $(0\leq k\leq n)$ be computed
by~\eqref{E:Def_h_k_Q_k} for a given $0\leq t\leq 1$. The point $\p{R}_n(u)$,
where $u\in[0,1]$, is in the convex hull of the points 
$\p{Q}_0, \p{Q}_1, \ldots, \p{Q}_n$ if and only if $u\leq t$. It means that
$$
\p{R}_n([0,u])\subset\mbox{conv}\{\p{Q}_0, \p{Q}_1, \ldots, \p{Q}_n\}\quad
(u\leq t).
$$
\end{theorem}

Let us notice that the proposed method can also be used for the
\textit{subdivision} of B\'{e}zier curve (cf., e.g., \cite[\S5.4]{Farin2002}).
For example, let us fix $u\in(0,1)$, it is well-known that the points
$$
\p{V}_k:=\sum_{j=0}^{k}B^k_j(u)\p{W}_j\quad (0\leq k\leq n)
$$
are the control points of the polynomial B\'{e}zier curve $\p{P}_n^L$ being the
\textit{left part} of the B\'{e}zier curve \eqref{E:BezierCurve} with 
$t\in[0,u]$. One can check that
$$
\p{V}_k=\sum_{j=0}^{k}h_j^{-1}\frac{n-k}{n-j}B^k_j(u)\p{Q}_j
                                                 \quad (0\leq k\leq n-1),
$$
$\p{V}_n=\p{Q}_n$, where the numbers $h_j$ and the points $\p{Q}_j$ 
$(0\leq j\leq n)$ are computed using \eqref{E:Def_h_k_Q_k} with $t:=u$,
$\omega_0=\omega_1=\ldots=\omega_n:=1$.

\subsection{Implementation and cost}           \label{SS:ImplementationCostBC}

Let us give efficient and numerically safe implementations of the new method
which have $O(dn)$ computational complexity and $O(1)$ memory complexity.

\begin{algorithm}[ht!]
\caption{First implementation}\label{A:Alg3}
\begin{algorithmic}[1]
\Procedure {NewRatBEval1}{$n, t, \omega, \p{W}$}


\State $h \gets 1$
\State $u \gets 1-t$
\State $n_1 \gets n+1$
\State $\p{Q} \gets \p{W}_0$


\For {$k \gets 1,n$}
       \State $h \gets h\cdot t\cdot(n_1-k)\cdot\omega_k$
       \State $h \gets h/(k\cdot u\cdot\omega_{k-1}+h)$
       \State $h_1 \gets 1-h$
       \State $\p{Q} \gets h_1\cdot\p{Q}+h\cdot\p{W}_k$
\EndFor


\State \Return $\p{Q}$


\EndProcedure
\end{algorithmic}
\end{algorithm}

The implementation provided in Algorithm \ref{A:Alg3} requires $(3d+8)n+1$
floating-point arithmetic operations (\textsf{flops}) to compute a point on 
a rational B\'{e}zier curve of degree $n$ in $\mathbb E^d$. 

\begin{algorithm}[ht!]
\caption{Second implementation}\label{A:Alg4}
\begin{algorithmic}[1]
\Procedure {NewRatBEval2}{$n, t, \omega, \p{W}$}


\State $h \gets 1$
\State $u \gets 1-t$
\State $n_1 \gets n+1$
\State $\p{Q} \gets \p{W}_0$


\If {$t\leq 0.5$}

\State $u \gets t/u$

\For {$k \gets 1,n$}
       \State $h \gets h\cdot u\cdot(n_1-k)\cdot\omega_k$
       \State $h \gets h/(k\cdot\omega_{k-1}+h)$
       \State $h_1 \gets 1-h$
       \State $\p{Q} \gets h_1\cdot\p{Q}+h\cdot\p{W}_k$
\EndFor

\Else

\State $u \gets u/t$

\For {$k \gets 1,n$}
       \State $h \gets h\cdot(n_1-k)\cdot\omega_k$
       \State $h \gets h/(k\cdot u\cdot\omega_{k-1}+h)$
       \State $h_1 \gets 1-h$
       \State $\p{Q} \gets h_1\cdot\p{Q}+h\cdot\p{W}_k$
\EndFor

\EndIf


\State \Return $\p{Q}$


\EndProcedure
\end{algorithmic}
\end{algorithm}

Algorithm~\ref{A:Alg4} decreases the number of \textsf{flops} to $(3d+7)n+2$.
However, for numerical reasons (cf.~lines 7 and 15 in Algorithm~\ref{A:Alg4}), 
it is necessary to use a~conditional statement. More precisely, one has to 
check whether $t\in[0,0.5]$ or $t\in(0.5,1]$, which can be easily done (it is
enough to check an exponent of a~floating-point number $t$).

Note that in the case of polynomial B\'{e}zier curves~\eqref{E:BezierCurve}, 
one only needs to set $\omega_k:=1$ $(0\leq k\leq n)$ in the given algorithms,
thus simplifying used formulas. Then the number of \textsf{flops} is equal to
$(3d+6)n+1$ in~Algorithm~\ref{A:Alg3} and $(3d+5)n+2$ 
in Algorithm~\ref{A:Alg4}.

\begin{table*}[ht!]
\begin{center}
\renewcommand{\arraystretch}{2.25}
\begin{tabular}{llcc}
&&
\begin{minipage}[t]{2.5cm}
new method\\
(cf.~Alg.~\ref{A:Alg4})   
\end{minipage} 
&
\begin{minipage}[t]{2.5cm}
de~Casteljau\\
(cf.~Appendix)\\ 
\end{minipage}\\ \hline
B\'{e}zier curve & 
in total & $(3d+5)n+2$ & $\displaystyle \frac{3dn(n+1)}{2}+1$ \\ 
& add/sub & $(d+2)n+1$ & $\displaystyle \frac{dn(n+1)}{2}+1$ \\ 
& mult & $2(d+1)n$ & $dn(n+1)$ \\ 
& div & $n+1$ & $0$ \\ \hline 
\begin{minipage}[t]{2.5cm}
rational B\'{e}zier\\ curve \\ \vspace{-1ex}
\end{minipage} 
& in total & $(3d+7)n+2$ & 
                  $\displaystyle \frac{(3d+5)n(n+1)}{2}+1$\\
& add/sub & $(d+2)n+1$ & 
            $\displaystyle \frac{(d+2)n(n+1)}{2}+1$\\ 
& mult & $2(d+2)n$ & $(d+1)n(n+1)$\\ 
& div & $n+1$ & $\displaystyle \frac{n(n+1)}{2}$\\ \hline 
\end{tabular}
\renewcommand{\arraystretch}{1}
\vspace{2ex}
\caption{
Numbers of \textsf{flops}.}\label{T:Table1}
\vspace{-3ex}
\end{center}
\end{table*}

The 
numbers of \textsf{flops} for the new algorithms, as well as for de Casteljau
algorithms (see Appendix), which also have a geometric interpretation and 
compute only convex combinations of control points, are given in
Table~\ref{T:Table1}.

\begin{example}\label{E:Example1}
Table~\ref{T:Table2} shows the comparison between the running times of de
Casteljau algorithm and Algorithm~\ref{A:Alg4} both for B\'{e}zier curves and
rational B\'{e}zier curves (in the case of B\'{e}zier curves,
Algorithm~\ref{A:Alg4} has been simplified), for $d\in\{2,3\}$. The results
have been obtained on a computer with 
\texttt{Intel Core i5-2540M CPU} at \texttt{2.60GHz} processor and \texttt{4GB}
\texttt{RAM}, using \texttt{GNU C Compiler 7.4.0} (single precision).

More precisely, we made the following numerical experiments. For a fixed $n$,
$10000$ curves of degree $n$ are generated. Their control points
$\p{W}_k\in[-1,1]^d$ and---in the rational case---weights $\omega_k\in[0.01,1]$
$(0\leq k\leq n)$ have been generated using the \texttt{rand()} C function. Each
curve is then evaluated at 501 points $t_i:=i/500$ $(0 \leq i \leq 500)$. Each
algorithm is tested using the same curves. Table~\ref{T:Table2} shows the total
running time of all $501\times 10000$ evaluations.
\end{example}

\begin{table*}[ht!]
\begin{center}
\renewcommand{\arraystretch}{1.45}
\begin{tabular}{lccccc}
 & & \multicolumn{2}{c}{B\'{e}zier curve} & 
\multicolumn{2}{c}{rational B\'{e}zier curve}\\
$n$ & $d$ &
\begin{minipage}[t]{2.5cm}
new method\\
(cf.~Alg.~\ref{A:Alg4})   
\end{minipage}
&
\begin{minipage}[t]{2.5cm}
de~Casteljau\\
(cf.~Appendix) 
\end{minipage}
&
\begin{minipage}[t]{2.5cm}
new method\\
(cf.~Alg.~\ref{A:Alg4})   
\end{minipage}
&
\begin{minipage}[t]{2.5cm}
de~Casteljau\\
(cf.~Appendix) 
\end{minipage}\\[3ex] \hline
1 & 2 & 2.654 & \textbf{2.603} & \textbf{2.672} & 2.685\\
 & 3 & 2.660 & \textbf{2.626} & \textbf{2.675} & 2.713 \\ \hline
2 & 2 & 2.760 & \textbf{2.694} & \textbf{2.773} & 2.846\\
 & 3 & \textbf{2.751} & 2.753 & \textbf{2.784} & 2.938 \\ \hline
3 & 2 & \textbf{2.848} & 2.854 & \textbf{2.872} & 3.070\\
 & 3 & \textbf{2.842} & 2.930 & \textbf{2.889} & 3.316 \\ \hline
4 & 2 & \textbf{2.954} & 3.028 & \textbf{2.989} & 3.380\\
 & 3 & \textbf{2.927} & 3.120 & \textbf{2.997} & 3.853 \\ \hline
5 & 2 & \textbf{3.017} & 3.238 & \textbf{3.094} & 3.868\\
 & 3 & \textbf{3.015} & 3.334 & \textbf{3.111} & 4.279 \\ \hline
6 & 2 & \textbf{3.113} & 3.551 & \textbf{3.203} & 4.201\\
 & 3 & \textbf{3.064} & 3.566 & \textbf{3.215} & 4.768 \\ \hline
10 & 2 & \textbf{3.468} & 4.447 & \textbf{3.593} & 6.549\\
 & 3 & \textbf{3.466} & 4.995 & \textbf{3.599} & 7.634 \\ \hline
15 & 2 & \textbf{3.830} & 6.194 & \textbf{4.057} & 10.979\\
 & 3 & \textbf{3.878} & 7.495 & \textbf{4.070} & 12.977 \\ \hline
20 & 2 & \textbf{4.241} & 8.461 & \textbf{4.527} & 17.261\\
 & 3 & \textbf{4.284} & 10.910 & \textbf{4.543} & 20.482 \\ \hline
\end{tabular}
\renewcommand{\arraystretch}{1}
\vspace{2ex}
\caption{Running times comparison (in seconds) for Example~\ref{E:Example1}. The
source code in C which was used to perform the tests is available at       
\texttt{http://www.ii.uni.wroc.pl/{\textasciitilde}pwo/programs/new-Bezier-eval-main.c}%
.}\label{T:Table2}
\vspace{-3ex}
\end{center}
\end{table*}

Observe that in the case of B\'{e}zier curves, the quantities $h$, which
are computed in the new algorithms, do not depend on the control points. One 
can use this fact in the fast evaluation of $M$ B\'{e}zier curves of the same 
degree $n$ for the same value of the parameter $t$. Such a~method requires
$(3dM+5)n+2$ \textsf{flops} while the direct use of the de Casteljau algorithm
means that all computations have to be repeated $M$ times, i.e., the number of
\textsf{flops} is equal to $3Mdn(n+1)/2+1$.

\begin{remark}
In rather rare cases $(h_k\approx1)$, the problem of \textit{cancellation of
digits} (\cite[\S2.3.4]{DB2008}) can occur while $1-h_k$ is computed 
(cf.~$h_1$ in~Algorithms~\ref{A:Alg3}, \ref{A:Alg4}). One can avoid this 
problem using the relation
$$
1-h_k=\frac{h_{k}}{h_{k-1}}\frac{\omega_{k-1}k(1-t)}{\omega_kt(n-k+1)}
\quad (1\leq k\leq n),
$$
if computations with high accuracy are necessary.
\end{remark}

\section{Remarks on evaluation of B\'{e}zier surfaces}     \label{S:NewEvalBS}

The method of evaluation described in Section~\ref{S:Introduction} can also
be applied to the rational rectangular and triangular B\'{e}zier surfaces.

Let $\p{S}_{mn}:[0,1]^2\rightarrow\mathbb E^d$ $(m,n,d\in\mathbb N)$ be
a \textit{rational rectangular B\'{e}zier surface} with the control points
$\p{W}_{ij}\in\mathbb E^d$ and weights $\omega_{ij}>0$ $(0\leq i\leq m,\ 
0\leq j\leq n)$,
$$
\p{S}_{mn}(s,t):=
  \frac{\displaystyle \sum_{i=0}^{m}
            \sum_{j=0}^{n}\omega_{ij}\p{W}_{ij}B^m_i(s)B^n_j(t)}
        {\displaystyle \sum_{i=0}^{m}
            \sum_{j=0}^{n}\omega_{ij}B^m_i(s)B^n_j(t)}.
$$ 

Define $T:=\{(s,t) : s,t\geq0,\,1-s-t\geq0\}$. Let there be given the control
points $\p{V}_{ij}\in\mathbb E^d$ and positive weights $v_{ij}$ 
$(0\leq i+j\leq n)$. Let $B^n_{ij}$ denotes the \textit{triangular Bernstein
polynomials}, 
$$
B^n_{ij}(s,t):=\frac{n!}{i!j!(n-i-j)!}s^it^j(1-s-t)^{n-i-j},
$$
where $0\leq i+j\leq n$. Let us consider a \textit{rational triangular 
B\'{e}zier surface} $\p{T}_n:T\rightarrow\mathbb E^d$ $(n,d\in\mathbb N)$ of 
the form
$$
\p{T}_n(s,t):=
  \frac{\displaystyle \sum_{i=0}^{n}
            \sum_{j=0}^{n-i}v_{ij}\p{V}_{ij}B^n_{ij}(s,t)}
        {\displaystyle \sum_{i=0}^{n}
            \sum_{j=0}^{n-i}v_{ij}B^n_{ij}(s,t)}.
$$

Both surface types are, in fact, \textit{rational parametric objects} 
(cf.~\eqref{E:ParObject}). Thus, one can apply Algorithm \ref{A:GenAlg} to
propose the methods which have geometric interpretations, compute only convex 
combinations of points and allow to evaluate B\'{e}zier surfaces in linear
time with respect to the number of control points, i.e., $O(nm)$ in the
\textit{rectangular} case and $O(n^2)$ in the \textit{triangular} case. To do
so, it is necessary to rearrange the sets of control points, corresponding
weights and basis functions (cf.~\eqref{E:BasisF}) into one-dimensional
sequences --- but since the method is agnostic of the ordering, the chosen
ordering is only a matter of preference. Taking into account that the
computations can be performed in many ways, we do not present \textit{rigorous}
algorithms and we pass some \textit{technical} details. 

In this section, to present a concise formulation of the methods, we
choose the \textit{row-by-row} order. For the reader's convenience, the 
analogues of quantities $h_k$ and points $\p{Q}_k$ from Algorithm 
\ref{A:GenAlg} have two indices instead, to correspond with the surfaces'
structure.

\subsection{Rational rectangular B\'{e}zier surfaces}        \label{SS:RBSurf}

Let $\p{S}_{mn}$ $(m,n\in\mathbb N)$ be a rational rectangular B\'{e}zier
surface with the weights $\omega_{ij}$ and control points 
$\p{W}_{ij}$ $(0\leq i\leq m,\, 0 \leq j \leq n)$.

In this case, one can interpret the set of control points as a rectangular grid
having $m+1$ rows with $n+1$ points in each row. We set the sequence of control
points so that:
\begin{itemize}
\itemsep1ex

\item the sequence begins with $\p{W}_{00}$,

\item $\p{W}_{i,j-1}$ is followed by $\p{W}_{ij}$ 
      $(0\leq i\leq m,\,1\leq j\leq n)$,

\item $\p{W}_{i-1,n}$ is followed by $\p{W}_{i0}$ $(1\leq i\leq m)$.
            
\end{itemize}
In a similar way, we set the sequences of weights $\omega_{ij}$ and basis
functions $B^m_i(s)B^n_j(t)$ $(0\leq i\leq m,\, 0 \leq j \leq n)$. 

It is well-known that if $(s,t)$ belongs to the boundary of the square 
$[0,1]^2$ then the point $\p{S}_{mn}(s,t)$ lies on the \textit{boundary}
rational B\'{e}zier curve with boundary control points and weights. Thus, the
method described in Section~\ref{SS:NewMethodBC} can be used in this case. 

Let us fix $(s,t)\in(0,1)^2$. Now, based on Algorithm \ref{A:GenAlg}, we define
the sequences of quantities $h_{ij}$ and points $\p{Q}_{ij}\in\mathbb E^d$ 
$(0\leq i\leq m,\,0 \leq j \leq n)$---determined in the order described 
above---in the following recurrent way:
$$
h_{ij}:=\left\{
\begin{array}{l}
1\quad(i=j=0),\\[1.5ex]
\left(1+\frac{\displaystyle 
                         i\omega_{i-1,n}(1-s)t^n}
                     {\displaystyle   
                         m_i\omega_{i0}h_{i-1,n}s(1-t)^n}\right)^{-1}\\[2ex]
\hfill(i\neq0,\,j=0),\\[1.25ex]                          
\left(1+\frac{\displaystyle j\omega_{i,j-1}(1-t)}
             {\displaystyle n_j\omega_{ij}h_{i,j-1}t}\right)^{-1}\\[2ex]             
\hfill(\mbox{otherwise}),                       
\end{array}
\right.
$$
$$
\hspace{-1.11cm}
\p{Q}_{ij}:=\left\{
\begin{array}{l}
\p{W}_{00}\quad(i=j=0),\\[0.75ex]
(1-h_{i0})\p{Q}_{i-1,n}+h_{i0}\p{W}_{i0}\\[0.75ex]                          
\hfill(i\neq0,\,j=0),\\[0.75ex]
(1-h_{ij})\p{Q}_{i,j-1}+h_{ij}\p{W}_{ij}\\[0.75ex]            
\hfill(\mbox{otherwise}),                                              
\end{array}
\right.
$$
where $0\leq i\leq m,\,0\leq j\leq n$, and $m_i:=m-i+1$, $n_j:=n-j+1$.

Theorem \ref{T:GenThm} implies that $\p{S}_{mn}(s,t)=\p{Q}_{mn}$. 

\subsection{Rational triangular B\'{e}zier surfaces}         \label{SS:TBSurf}

Suppose $\p{T}_n$ $(n\in\mathbb N)$ is a rational triangular B\'{e}zier
surface associated with the weights $v_{ij}$ and control points $\p{V}_{ij}$ 
$(0\leq i+j \leq n)$.

The method described below is analogous to the one for rectangular B\'{e}zier
surfaces. The main difference is that, in this case, the set of the control
points can be seen as a triangular grid, i.e., the number of control points in
each row depends on the row number. Namely, there are $n-i+1$ points in the
$i$th row $(0\leq i\leq n)$ of this triangular grid. We choose the following
ordering of control points:
\begin{itemize}
\itemsep1ex

\item the sequence begins with $\p{V}_{00}$,

\item $\p{V}_{i,j-1}$ is followed by $\p{V}_{ij}$ 
      $(0\leq i\leq n-1,\,1\leq j\leq n-i$),

\item $\p{V}_{i-1,n-i+1}$ is followed by $\p{V}_{i0}$ $(1\leq i\leq n)$.
      
\end{itemize}
We set the sequences of weights $v_{ij}$ and basis functions $B^n_{ij}(s,t)$
$(0\leq i+j \leq n)$ in the same way.

Assume $(s,t)$ is on the boundary of the triangle $T$. Then the point 
$\p{T}_n(s,t)$ lies on the \textit{boundary} rational B\'{e}zier curve having 
known control points and weights and, again, one can compute this point using
the method presented in Section~\ref{SS:NewMethodBC}.

Let us fix a point $(s,t)$ inside the triangle $T$. Similarly, based on 
Algorithm~\ref{A:GenAlg}, we introduce the sequences of quantities $g_{ij}$ and
points $\p{U}_{ij}\in\mathbb E^d$ $(0\leq i+j \leq n)$, which are computed in
the order described above, by the following recurrent formulas:
$$
g_{ij}:=\left\{
\begin{array}{l}
1\quad(i=j=0),\\[1.5ex]
\left(1+\frac{\displaystyle iv_{i-1,n-i+1}t^{n-i+1}}
             {\displaystyle n_iv_{i0}g_{i-1,n-i+1}sr^{n-i}}\right)^{-1}\\[2ex]
\hfill(i\neq0,\,j=0),\\[1.25ex]                          
\left(1+\frac{\displaystyle jv_{i,j-1}r}
             {\displaystyle n_{i+j}v_{ij}g_{i,j-1}t}\right)^{-1}\\[2ex]             
\hfill(\mbox{otherwise}),                       
\end{array}
\right.\hfill
$$
$$
\hspace{-0.48cm}
\p{U}_{ij}:=\left\{
\begin{array}{l}
\p{V}_{00}\quad (i=j=0),\\[0.75ex]
(1-g_{i0})\p{U}_{i-1,n-i+1}+g_{i0}\p{V}_{i0}\\[0.75ex]                          
\hfill(i\neq0,\,j=0),\\[0.75ex]
(1-g_{ij})\p{U}_{i,j-1}+g_{ij}\p{V}_{ij}\\[0.75ex]            
\hfill(\mbox{otherwise}),                                              
\end{array}
\right.\hfill
$$
where $0\leq i\leq n,\,0\leq j\leq n-i$, and $r:=1-s-t$, $n_l:=n-l+1$.

Then $\p{T}_n(s,t)=\p{U}_{n0}$, which follows from Theorem \ref{T:GenThm}.

\renewcommand{\thesection}{\Alph{section}}
\setcounter{section}{1}
\setcounter{algorithm}{0}
\section*{Appendix. Implementations of de Ca\-stel\-jau algorithms for 
          B\'{e}zier curves}                                                           

For the reader's convenience, let us also present efficient implementations
of de Casteljau algorithms 
which have $O(n)$ memory complexity. See Algorithms \ref{A:Alg6} and 
\ref{A:Alg5} (cf., e.g., \cite{Farin2002}). The 
numbers of \textsf{flops} for these methods are given in Table~\ref{T:Table1}.

\begin{algorithm}[ht!]
\caption{De Casteljau algorithm}\label{A:Alg6}
\begin{algorithmic}[1]
\Procedure {BEval}{$n, t, \p{W}$}


\State $t_1 \gets 1-t$


\For {$i \gets 0,n$}
       \State $\p{Q}_{i} \gets \p{W}_{i}$
\EndFor


\For {$k \gets 1,n$}
  \For {$i \gets 0,n-k$}
       \State $\p{Q}_{i} \gets t_1\cdot\p{Q}_{i}+t\cdot\p{Q}_{i+1}$
  \EndFor
\EndFor


\State \Return $\p{Q}_{0}$


\EndProcedure
\end{algorithmic}
\end{algorithm}

\begin{algorithm}[ht!]
\caption{Rational de Casteljau algorithm}\label{A:Alg5}
\begin{algorithmic}[1]
\Procedure {RatBEval}{$n, t, \omega, \p{W}$}
             

\State $t_1 \gets 1-t$


\For {$i \gets 0,n$}
       \State $w_{i} \gets \omega_{i}$
       \State $\p{Q}_{i} \gets \p{W}_{i}$
\EndFor


\For {$k \gets 1,n$}
  \For {$i \gets 0,n-k$}
       \State $u \gets t_1\cdot w_{i}$
       \State $v \gets t\cdot w_{i+1}$
       \State $w_{i} \gets u+v$
       \State $u \gets u/w_{i}$
       \State $v \gets 1-u$
       \State $\p{Q}_{i} \gets u\cdot\p{Q}_{i}+v\cdot\p{Q}_{i+1}$
  \EndFor
\EndFor


\State \Return $\p{Q}_{0}$


\EndProcedure
\end{algorithmic}
\end{algorithm}

\newpage
\bibliographystyle{elsart-num-sort}
\bibliography{rev-3-new-Bezier-eval-two-column}

\begin{thebibliography}{1}
\expandafter\ifx\csname url\endcsname\relax
  \def\url#1{\texttt{#1}}\fi
\expandafter\ifx\csname urlprefix\endcsname\relax\def\urlprefix{URL }\fi

\bibitem{Bezerra2013}
L.~Bezerra, Efficient computation of {B}\'{e}\-zier curves from their
  {B}ernstein-{F}ourier representation, Applied Mathematics and Computation 220
  (2013) 235--238.

\bibitem{DB2008}
G.~Dahlquist, {\AA}.~Bj{\"o}rck, Numerical methods in scientific computing.
  {V}ol.~{I}, SIAM, Philadelphia, 2008.

\bibitem{Farin2002}
G.~Farin, Curves and surfaces for com\-puter-aided geometric design. A
  practical guide, 5th ed., Academic Press, Boston, 2002.

\bibitem{Peters1994}
J.~Peters, Evaluation and approximate evaluation of the multivariate
  {B}ernstein-{B}\'{e}zier form on a regularly partitioned simplex, ACM
  Transactions on Mathematical Software (TOMS) 20~(4) (1994) 460--480.

\bibitem{SV1986}
L.~Schumaker, W.~Volk, Efficient evaluation of multivariate polynomials,
  Computer Aided Geometric Design 3 (1986) 149--154.

\end{thebibliography}


\end{document}